\documentclass[12pt]{amsart}
\usepackage{amsmath,amssymb,txfonts}
\usepackage{times}
\usepackage{bbm}
\usepackage[dvips]{graphicx}
\allowdisplaybreaks
\usepackage{amsxtra}
\usepackage[cp1252]{inputenc}

\usepackage{mathrsfs}

\numberwithin{equation}{section}

\newcommand{\C}{{\mathbb C^n}}

\newcommand{\R}{\mathbb R}

\newtheorem{theorem}{Theorem}[section]

\newtheorem{Remark}[theorem]{Remark}
\newtheorem{conjecture}[theorem]{Conjecture}
\newtheorem{BVP}[theorem]{Problem BVP}

\newtheorem{Corollary}[theorem]{Corollary}

\newtheorem{lemma}[theorem]{Lemma}

\newcommand{\ct}[1]{\langle {#1}\rangle \lower.3ex\mathrm{$_{t}$}}
\newcommand{\lt}[1]{[ {#1}] \lower.3ex\mathrm{$_{t}$}}

\numberwithin{equation}{section}

\begin{document}

\title[The minimal surface flow with prescribed contact angle]{The global solution of the minimal surface flow and translating surfaces}
\author{Li Ma, YuXin Pan}

\address{Li MA,  School of Mathematics and Physics\\
  University of Science and Technology Beijing \\
  30 Xueyuan Road, Haidian District
  Beijing, 100083\\
  P.R. China }
 \email{lma17@ustb.edu.cn}

 \address{YuXin Pan,  School of Mathematics and Physics\\
  University of Science and Technology Beijing \\
  30 Xueyuan Road, Haidian District
  Beijing, 100083\\
  P.R. China }
 \email{yxinpan@163.com}


\begin{abstract}
In this paper, we study evolved surfaces over convex planar domains which are evolving by the minimal surface flow
$$u_{t}= div\left(\frac{Du}{\sqrt{1+|Du|^2}}\right)-H(x,Du).$$
Here, we specify the angle of contact of the evolved surface to the boundary cylinder. The interesting question is to find translating solitons of the form $u(x,t)=\omega t+w(x)$ where $\omega\in \mathbb R$. Under an angle condition, we can prove the a priori estimate holds true for the translating solitons (i.e., translator), which makes the solitons exist.
We can prove for suitable condition on $H(x,p)$ that there is the global solution of the minimal surface flow.
Then we show, provided the soliton exists, that the global solutions converge to some translator.



\end{abstract}

\maketitle

\tableofcontents

\section{Introduction}\label{sect1}
Motivated by recent works from mean curvature flow and minimal surface flow (see Ecker \cite{Ec} and Altschuler-Wu \cite{AW}), we consider the minimal surface flow on a bounded planar domain with prescribed contact angle on the boundary. Throughout this paper, we let  $\Omega\subset R^2$ be a compact planar convex domain with smooth boundary $\partial\Omega$ and we denote by $k>0$ the curvature of the boundary $\partial\Omega$. The inward pointing normal and counterclockwise tangent vector to $\partial\Omega$ will be denoted by N and T. The upward normal for a graph $u:\Omega\rightarrow R^1$ is $\gamma=(-Du,1)/\sqrt{1+|Du|^2}$. The angle of contact between the graph and the boundary, $\alpha:\partial\Omega\rightarrow(0,\pi)$, is given by $\langle\gamma,N\rangle=\cos\alpha$ or $D_{N}u=-\cos\alpha\sqrt{1+|Du|^2}$. Given a smooth function $H:\overline{\Omega}\times \mathbb R^2\to \mathbb R$ and $T>0$. We consider the following initial-boundary value problem
\begin{equation}\label{eq1_1}
\left\{
\begin{aligned}
& u_{t}= div\left(\frac{Du}{\sqrt{1+|Du|^2}}\right)-H(x,D u) & on \ Q_{T} \\
& D_{N}u=-\cos\alpha\sqrt{1+|Du|^2} & on \ \Gamma_{T} \\
& u(x,0)=u_{0}(x) & on \ \Omega
\end{aligned}
\right.
\end{equation}
where $u_{0}\in C^{\infty}(\bar{\Omega})$ is the initial data satisfying $D_{N}u_{0}=-\cos\alpha\sqrt{1+|Du|^2}$ on $\partial\Omega$. We denote by
$$Q_{T}=\Omega\times[0,T);\quad \Gamma_{T}=\partial\Omega\times[0,T);\quad \Omega_{t}=\Omega\times \left\{ t \right\}.$$
The motivation for such a research came from the interest of the Capillary Surfaces. 
To our best knowledge, Lichnewsky and Temam \cite{Temam1978} \cite{Lich} had first studied generalized solutions of
the evolutionary minimal surface equation \eqref{eq1_1} (which we call the minimal surface flow). Then C.Gerhardt \cite{Ge} and K.Ecker \cite{Ec} had studied the evolutionary minimal surface equation and obtained very interesting gradient estimates for classical solutions.
To our surprise, there is not much works about the problem \eqref{eq1_1}. We refer to the related references in \cite{Ec} \cite{Ge,Ge2}.
We want to show that there is a priori estimate for the problem \eqref{eq1_1} so that we may get a global solution to the flow problem.
To study the behavior of the global flow, we need to study the translating solutions (i.e., translator) in the form $u(x,t)=w(x)+Ct$ to the problem \eqref{eq1_1} so that we have
corresponding elliptic problem as below
\begin{equation}\label{ell-eq1_2}
\left\{
\begin{aligned}
&  div\left(\frac{Du}{\sqrt{1+|Du|^2}}\right)-H(x,D u)=C, & on \ \Omega \\
& D_{N}u=-\cos\alpha\sqrt{1+|Du|^2} & on \ \partial \Omega.
\end{aligned}
\right.
\end{equation}
Note that adding an extra constant to the solution $u$ of the problem \eqref{ell-eq1_2}, it is still a solution to the problem \eqref{ell-eq1_2}.
Again, by assuming for any $(x,p)\in \Omega\times \mathbb R^2$,
\begin{equation}\label{ell-eq1_3}
k-|D_{T}\alpha|-|C+H(x,p)|\geq\delta_{0}>0;\quad k_{0}\geq k>0;\quad |\alpha|\leq\alpha_{0}<\pi,
\end{equation}
  we have a priori estimate for the problem \eqref{ell-eq1_2} so that we may get a smooth solution to the elliptic problem via the continuity method. Then we have the following result.
  \begin{theorem}\label{ell-thm1_1}
 Assume $p\cdot H_x(x,p)\geq 0$ and for some constant $c_0>0$, $|C+H(x,p)|\leq c_0$ for any $(x,p)\in \Omega\times \mathbb R^2$.
 We further assume that $\exists k_{0},\alpha_{0},\delta_{0} \in \mathbf{R}^{+}$ such that
\begin{equation}\label{eq1_4}
k-|D_{T}\alpha|-c_0\geq\delta_{0}>0;\quad k_{0}\geq k>0;\quad |\alpha|\leq\alpha_{0}<\pi,
\end{equation}
 Then for any smooth solution $u$ to \eqref{ell-eq1_2}, $\exists c_{1}=c_{1}(\alpha_{0},\delta_{0},k_{0},u_{0})$ so that $|Du|^2 \leq c_{1}$ on $\Omega$, and thus $u(x) \in C^{\infty}(\Omega)$.
\end{theorem}

  Using the apriori estimate above we can get the existence of the solution to the problem \eqref{ell-eq1_2} as in the proof of Theorem 2.1 in  \cite{Ge2} and by now it is standard and we may omit the detail. One may see \cite{MS} for the one dimensional case. We remark that our continuity method is based on the solution constructed in \cite{AW}. We may take $0\in \Omega$. Namely we consider the problem
  $$
  div\left(\frac{Du}{\sqrt{1+|Du|^2}}\right)=(1-s)\frac{c_1}{\sqrt{1+|Du|^2}}+s(H(x,D u)+C),
  $$
  with the boundary condition
  $$
  D_{N}u=-\cos\alpha\sqrt{1+|Du|^2}, \ \  on \ \partial \Omega,
  $$
  with $s\in [0,1]$. When $s=0$, there is a unique solution constructed by Altschular-Wu \cite{AW} with the normalization condition that $u(0)=0$. Then once we have the uniform gradient estimate, we have the uniform estimate of the solution $u$. Note that in our case we can not directly invoke the uniform estimate of the solution $u$ from the work of Concus-Finn\cite{CF} ( see also \cite{U}).

  We want to study the asymptotic behavior of the global solution of the problem \eqref{eq1_1} and we can show the profile of the flow is some translating soliton from the problem \eqref{ell-eq1_2}. As we shall see soon, we do obtain some new results for the problem \eqref{eq1_1}.

As we mentioned above, the problem similar to \eqref{eq1_1} has been studied by famous researchers such as C.Gerhardt and K.Ecker.
In \cite{Ec} Ecker studied the following parabolic equation
\begin{equation*}
\left\{
\begin{aligned}
& u_{t}= div\left(\frac{Du}{\sqrt{1+|Du|^2}}\right)-H(x,u) &in& \ \Omega\times(0,T)\\
& u(x,0)=u_{0}(x) &in& \ \Omega\times\{0\}
\end{aligned}
\right.
\end{equation*}
and proved that $|Du| \leq c\cdot d^{-2}$ provided the function $H$ is monotone increasing in $u$. One may also refer to the works of Concus-Finn \cite{CF}, etc, for the related elliptic problem and one may consult Lieberman's book \cite{Li2} for more references. We remark that our existence result for the problem \eqref{ell-eq1_2} may be different from existence results obtained in \cite{Li2} (see also \cite{MWW}) and the advantage of our result may be that our assumption is more geometric.

The closely related problem to the problem \eqref{eq1_1} is the mean curvature flow, which has less nonlinear than the problem \eqref{eq1_1} in some sense.
When $n\geq2$, $\varphi(x)=0$, Huisken \cite{Hu} studied the mean curvature flow
\begin{equation}\label{eq1_5}
\left\{
\begin{aligned}
& u_{t}= \sqrt{1+|Du|^2}div\left(\frac{Du}{\sqrt{1+|Du|^2}}\right) &in& \ \Omega\times[0,\infty)\\
& D_{N}u=\psi(x) &on& \ \partial\Omega\times[0,\infty)\\
& u(x,0)=u_{0}(x) &on& \ \Omega,
\end{aligned}
\right.
\end{equation}
and he proved that the solutions asymptotically converge to constant functions. For $n=2$, Altschuler and Wu \cite{AW} studied the equation \eqref{eq1_5} with $\varphi(x)=-\cos\alpha\sqrt{1+|Du|^2}$ and proved that the solutions of mean curvature flow converges to a surface which moves at a constant speed under some prescribed conditions.
For $n\geq2$, Guan \cite{Gu} studied the more generalized mean curvature type evolution equation with the prescribed contact angle problem as below
\begin{equation*}
\left\{
\begin{aligned}
& u_{t}= \sqrt{1+|Du|^2}div\left(\frac{Du}{\sqrt{1+|Du|^2}}\right)+\phi(x,u,Du) &in& \ \Omega\times (0,\infty) \\
& \langle -\gamma,N \rangle=\varphi(x) &on& \ \partial\Omega\times (0,\infty) \\
& u(x,0)=u_{0}(x) &on& \ \Omega.
\end{aligned}
\right.
\end{equation*}
He proved that the solution asymptotically approaches the solution to the corresponding stationary equation as $t\rightarrow\infty$ when the exterior force term $\phi(x,u,Du)=-ku\sqrt{1+|Du|^2}$ for $k>0$ or $\phi(x,u,Du)=n/u$ for $u>0$. As a consequence of the latter case, he obtained some existence results for minimal surfaces in hyperbolic space $\mathbb{H}^{n+1}$ with a prescribed contact angle condition.
Recently, Ma \cite{Ma} studied properties of the potential function of a translating soliton. Some interesting conclusions about the existence and the regularity of solution have been obtained, see, e.g.\cite{MY,Zh,Ge,Ge2,CGLM,Li,BNO}.\\

It should be interesting to investigate the nonparametric problem \eqref{eq1_1} with the nontrivial term $H(x,Du)$ and for this flow we can prove the following result.

\begin{theorem}\label{thm1_2}
 Assume that the initial quantity $|u_t(x,0)|$ is bounded on $\overline{\Omega}$. Then we have a local in time solution to \eqref{eq1_1} with the initial data $u_0$. Assume that there exists some uniform constant $c_1\geq 0$ such that $|u_t(x,0)|+|H(x,p)|\leq c_1$ on $\Omega\times \mathbb R^2$ and we further assume that $\exists k_{0},\alpha_{0},\delta_{0} \in \mathbf{R}^{+}$ such that
\begin{equation}\label{eq1_6}
k-|D_{T}\alpha|-c_1\geq\delta_{0}>0;\quad k_{0}\geq k>0;\quad |\alpha|\leq\alpha_{0}<\pi
\end{equation}
and $p\cdot H_x(x,p)\geq 0$ on $\Omega\times \mathbb R^2$, then the solution is a global one and $\exists c_{2}=c_{2}(\alpha_{0},\delta_{0},k_{0},u_{0})$ so that $|Du|^2 \leq c_{2}$ on $Q_{\infty}$, and thus $u(x,t) \in C^{\infty}(Q_{\infty})$.
\end{theorem}

As we shall see that the understanding of the
 elliptic version of \eqref{eq1_1}
\begin{align*}
& \frac{1}{\sqrt{1+|Dw|^2}}a^{ij}(Dw) D_i D_j w = C +H(x,Du),\ &on&\ \Omega\\
& D_N w = -\cos\alpha \sqrt{1+|Dw|^2} \ &on&\ \partial\Omega.
\end{align*}
is important for us to study the asymptotic behavior of the global solution.
Here $C$ is a constant determined uniquely by
$
 \frac{\int_{\partial\Omega} \cos\alpha ds-\int_\Omega H(x,Dw)dx}{|\Omega|}.
$
Actually, to study the asymptotic behavior of the global solutions, we need results from the corresponding elliptic problem.

In short, sometimes we denote by $H=H(x,p)$. We have the following result.
\begin{theorem}\label{thm1_3} Let $\Omega$ be a bounded convex domain in the plane and $H=H(x)$ is a smooth function on $\overline{\Omega}$ such that the problem \eqref{ell-eq1_2} has a solution.  Let
$u(x,t)$ be the  global solution with the initial data $u_0$ to the problem \eqref{eq1_1} and let $\alpha\in (0,\pi)$ for the initial data $u_0$ for the problem \eqref{eq1_1}. Then the global solution
$u(x,t)$ converges as $t\rightarrow \infty$ to a surface $u_{\infty}$ with mean curvature $H(x)$ (unique up to translation) which moves at a constant speed $C$ (uniquely determined by the boundary data).
\end{theorem}

In similar way, we may have the result below.
\begin{theorem}\label{thm1_4} Assume $H=H(x)$ as above. Let $\alpha$ be the angle function such that elliptic version of the problem \eqref{eq1_1} such that it has a solution
with  the domain $\Omega$ being a bounded convex domain in the plane.
If $\displaystyle \int_{\partial\Omega} \cos\alpha -\int_{\Omega} H(x)= 0$, then $C = 0$, hence $u_{\infty}$ is a surface with mean curvature $H(x)$.
\end{theorem}

Similar results about higher dimensional minimal surface flow with prescribed contact angle has been obtained in \cite{MP}.

This paper is organized as follows. Section \ref{sect2} is concentrated on some notations and the elliptic problem on bounded convex domains. The key time derivative and gradient estimate to solutions of \eqref{eq1_1} is given in section \ref{sect3}. In section \ref{sect4},  we prove Theorem \ref{thm1_2} and Theorem \ref{thm1_3} .

\section{Boundary geometry and boundary condition}\label{sect2}
  In this section we review the boundary geometry and boundary value analysis from \cite{AW}.
For convenience, we recall some notations from \cite{AW}. We let
\begin{eqnarray*}
\displaystyle v&=&\sqrt{1+|Du|^2},\\
\displaystyle a^{ij}&=&\delta_{ij} - \frac{D_{i}u D_{j}u}{1+|Du|^2},
\end{eqnarray*}
and
$\displaystyle u_{t}=\frac{\partial u}{\partial t}$.
Recall that the mean curvature of $\mathbb H$ of graph of $u$ is
\begin{eqnarray*}
\mathbb H = div\left(\frac{Du}{\sqrt{1+|Du|^2}}\right) = \partial_{i}\left(\frac{u_i}{\sqrt{1+|Du|^2}}\right) = \frac{1}{v}a^{ij}u_{ij},
\end{eqnarray*}
and the curvature norm $|A|$ of graph of $u$ is given by
\begin{eqnarray*}
|A|^2 = \frac{1}{v^2}a^{ij}a^{kl}u_{il}u_{jk}.
\end{eqnarray*}
For vectors $V$, $W$ we make the formula as in \cite{AW} that
\begin{eqnarray}\label{eq2_1}
D_{V}D_{W}u = V^{i}W^{j}D_{ij} ^{2}u + \langle D_{V}W,Du \rangle.
\end{eqnarray}
On the boundary $\partial\Omega$, we define
\begin{eqnarray*}
\displaystyle &a^{TN}&=a^{ij}T_{i}N_{j} = -\frac{D_T u D_N u}{1+|Du|^2} = a^{NT},\\
\displaystyle &a^{TT}&=a^{ij}T_{i}T_{j} = 1-\frac{|D_T u|^2}{1+|D u|^2} = \frac{1+|D_N u|^2}{1+|D u|^2},\\
\end{eqnarray*}
and
$$
\displaystyle a^{NN}=a^{ij}N_{i}N_{j} = 1-\frac{|D_N u|^2}{1+|D u|^2} = \frac{1+|D_T u|^2}{1+|D u|^2}.
$$

We also recall some  calculation rules  based on the matrix $(a_{ij})$, which are very useful
 formulae on the boundary so that we can make the uniform gradient estimates only by boundary data.

On $\partial\Omega$, we make smooth extensions of $N$ and $T$ to a tubular neighborhood of $\partial\Omega$. Let $\displaystyle \left\{ \frac{\partial}{\partial r},\frac{\partial}{\partial\theta} \right\}$ be the local frames on the tubular neighborhood of $\partial\Omega$ that $\displaystyle \left\langle \frac{\partial}{\partial r},\frac{\partial}{\partial \theta}\right\rangle = 0$, $\displaystyle \left|\frac{\partial}{\partial r}\right|^2=1$, and $\displaystyle \left\{ \frac{\partial}{\partial r},\frac{\partial}{\partial\theta} \right\}_{\partial\Omega} = \{N,T\}$.
 The following lemma  is from the work \cite{AW} (see 2.1. Lemma there). We present the full proof for completeness.

\begin{lemma}\label{lem2_1}
 Define a function $\displaystyle \phi$ such that $\displaystyle \left|\phi^{-1} \frac{\partial}{\partial \theta}\right|^2 = 1$ so that we have $\displaystyle \{N,T\} = \left\{\frac{\partial}{\partial r},\phi^{-1} \frac{\partial}{\partial \theta}\right\}$. Then we have the following two assertions:

1. for any $h \in C^{\infty}(\bar{\Omega})$, the interchange of derivatives is given by $D_N D_T h = D_T D _N h + kD_T h$;

2. $D_T T = kN$; $\quad D_T N = -k T$; \quad $D_N T = D_N N =0$.
\end{lemma}

\begin{proof}
At $\partial\Omega$, the Frenet formula tells us $D_{T}T = kN$ and $D_{T}N = -kT$. Since $|\phi^{-1}\frac{\partial}{\partial \theta}|^2 =1$, we get $\displaystyle \phi^2 = \left\langle \frac{\partial}{\partial \theta},\frac{\partial}{\partial \theta} \right\rangle$. Thus, as in \cite{AW},
\begin{eqnarray*}
\frac{\partial}{\partial r}\phi^2 &=&\frac{\partial}{\partial r}\langle \frac{\partial}{\partial \theta},\frac{\partial}{\partial \theta} \rangle,\\
2\phi \phi_r &=& 2\langle \frac{\partial}{\partial r}\frac{\partial}{\partial \theta},\frac{\partial}{\partial \theta} \rangle,\\
\phi \phi_r &=& \phi^2 \langle \phi^{-1}\frac{\partial}{\partial \theta}\frac{\partial}{\partial r},\phi^{-1}\frac{\partial}{\partial \theta} \rangle\\
 &=& \phi^2 \langle D_T N,T \rangle\\
 &=& -k\phi^2, \ \ \Rightarrow \ \phi_r=-k\phi.
\end{eqnarray*}
For any $h\in C^\infty$, we have
\begin{eqnarray*}
\displaystyle D_{N}D_{T}h&=&\frac{\partial}{\partial r}\left(\phi^{-1} \frac{\partial h}{\partial \theta}\right)\\
\displaystyle &=&-\phi^{-2}\frac{\partial \phi}{\partial r}\frac{\partial h}{\partial \theta} + \phi^{-1}\frac{\partial}{\partial \theta}\frac{\partial h}{\partial r}\\
\displaystyle &=&k\phi^{-1}\frac{\partial h}{\partial \theta}+\phi^{-1}\frac{\partial}{\partial \theta}\frac{\partial h}{\partial r}\\
\displaystyle &=&k D_{T}h+D_{T}D_{N}h,
\end{eqnarray*}
which is the first assertion and it can be written as
$$ [D_N,D_T]h=kD_Th.
$$ We now prove the second assertion and we only need to prove that $D_N T =0$ and $D_{N}N=0$. Using \eqref{eq2_1}, we have
$$
D_ND_Th=N^{i}T^{j}D^{2}_{ij}h + \langle D_{N}T,Dh \rangle = T^{i}N^{j}D^{2}_{ij}h + \langle D_{T}N,Dh \rangle + kD_T h.$$
Re-arranging the order of differentiation we have
$$N^{i}T^{j}D^{2}_{ij}h = T^{i}N^{j}D^{2}_{ij}h.$$
Hence,
$$\langle D_{N}T,Dh \rangle = \langle D_{T}N,Dh \rangle + kD_{T}h.$$
Since $D_T N = -kT$, we have $\langle D_{N}T,Dh \rangle =0$.
For arbitrariness of $h$, we have $D_N T =0$. By $\langle N,N \rangle =1$ and $\langle N,T \rangle =0$, we have
\begin{eqnarray*}
D_N \langle N,N \rangle &=& 2\langle D_N N, N \rangle = 0, \\
D_N \langle N,T \rangle &=& 0.
\end{eqnarray*}
Note that
$$
D_N \langle N,T \rangle=\langle D_N N, T \rangle + \langle D_N T,N \rangle= \langle D_N N, T \rangle.
$$
Thus, $\langle D_N N, T \rangle= 0$ and $D_{N}N = \langle D_N N, N\rangle +\langle D_N N, T \rangle = 0$.
\end{proof}
By the boundary condition for the problem \eqref{eq1_1} we have
\begin{eqnarray}\label{eq2_2}
\displaystyle D_{N}u&=&-\cos\alpha\sqrt{1+|Du|^2},
\end{eqnarray}
which implies that
\begin{eqnarray}\label{eq2_3}
\displaystyle |D_{N}u|^2&=&\cos^2\alpha(1+|Du|^2),
\end{eqnarray}
and
\begin{eqnarray}\label{eq2_4}
\displaystyle |D_{T}u|^2&=&\sin^2\alpha(1+|Du|^2)-1.
\end{eqnarray}
\noindent Differentiating conditions \eqref{eq2_2}-\eqref{eq2_4} in the time and tangential direction respectively. We have
\begin{eqnarray}\label{eq2_5}
\displaystyle D_{N}u_{t} = -\cos\alpha\frac{Du Du_{t}}{\sqrt{1+|Du|^2}},
\end{eqnarray}
\begin{eqnarray}\label{eq2_6}
\displaystyle D_{T}D_{N}u = \sin\alpha(D_{T}\alpha)v-\cos\alpha(D_{T}v),
\end{eqnarray}
\begin{eqnarray}\label{eq2_7}
\displaystyle D_{N}D_{T}u = \sin\alpha(D_{T}\alpha)v-\cos\alpha(D_{T}v)+k D_{T}u,
\end{eqnarray}
\begin{eqnarray}\label{eq2_8}
\displaystyle D_{T}D_{T}u = \frac{\sin\alpha\cos\alpha(D_{T}\alpha)v^2+\sin^2\alpha v D_{T}v}{D_{T}u}.
\end{eqnarray}
The formulae above will play an important role for our gradient estimate of the solution $u$.

\section{Elliptic problem}\label{sect3}
We now consider the \emph{Elliptic problem} \eqref{ell-eq1_2}.
We let $\alpha\in (0, \pi)$ and let $\Omega$ be a convex domain in the plane. In the argument below, we keep to use $\alpha$ to indicate the argument works well for general angles.
\begin{BVP}\label{BVP3_1}
Recall the elliptic version of \eqref{ell-eq1_2} is
\begin{equation}\label{eq3_8}
\left\{
\begin{aligned}
& \frac{1}{\sqrt{1+|Dw|^2}}a^{ij}(Dw) D_i D_j w = C+H(x, Dw) \ &on&\ \Omega\\
& D_N w = -\cos\alpha\sqrt{1+|Dw|^2} \ &on&\ \partial\Omega.
\end{aligned}
\right.
\end{equation}
where $C$ is a constant determined uniquely by \eqref{eq3_9} below.
\end{BVP}
Since the left-hand side of the first equation in \eqref{eq3_8} can be represented as $\displaystyle  div\left( \frac{Dw}{\sqrt{1+|Dw|^2}}\right)$. By applying divergence theorem, we have
$$
\int_{\Omega} C+H(x,Dw) dx = \int_{\Omega} div\left( \frac{Dw}{\sqrt{1+|Dw|^2}}\right) dx,
$$
and
\begin{equation}\label{eq3_9}
\begin{aligned}
C &= \frac{\int_{\Omega} div (\frac{Dw}{\sqrt{1+|Dw|^2}})dx -\int_{\Omega} H(x,Dw) dx}{\int_{\Omega} dx} \\
&= \frac{\int_{\partial \Omega} \frac{Dw}{\sqrt{1+|Dw|^2}}\cdot(-N)ds-\int_{\Omega} H(x,Dw) dx}{|\Omega|} \\
&= \frac{\int_{\partial\Omega} \cos\alpha ds-\int_{\Omega} H(x,Dw) dx}{|\Omega|} .
\end{aligned}
\end{equation}

\begin{theorem}\label{ell-thm3-2}
Under the assumption \eqref{ell-eq1_3} and the assumption
\begin{equation}\label{ell-eq1_3}
p\cdot H_x(x,p)\geq 0, \ \ |H(x,p)|\leq c,
\end{equation}
 we have, for the solution $u$ to the problem BVP \ref{BVP3_1}, some uniform constant $c_{1}=c_{1}(\alpha_{0},\delta_{0},k_{0})>0$ such that
$$ \sup_{\Omega} v \leq c_{1}.
$$
\end{theorem}
\begin{proof}
We plan to show that the maximum must occur on $\partial \Omega$ and then we use the boundary condition to get the uniform gradient estimate.
Note that
$$
0=\partial_{i}\left(\frac{1}{v}a^{ij}\left(u_{k}\right)_{j}\right)-(H_{x_k}+H_p(x,Du)Du_k).
$$
Multiplying both sides by $\frac{1}{v} u_k $ we derive
\begin{eqnarray*}
0
&=& \frac{1}{v} u_k \left(\frac{1}{v} a^{ij} u_{kj}\right)_i- \frac{1}{v} u_k(H_{x_k}+H_p(x,Du)Du_k)\\
&=& \frac{1}{v} \left( \frac{1}{v}a^{ij} u_k u_{kj}\right)_i - \frac{1}{v^2}a^{ij}u_{ki}u_{kj}- \frac{1}{v} u_k(H_{x_k}) - H_p(x,Du)Dv\\
&=& \frac{1}{v} \left( a^{ij} v_j\right)_i - |A|^2- \frac{1}{v} u_k(H_{x_k})-H_p(x,Du)Dv,
\end{eqnarray*}
which is the elliptic equation for $v$.
Since all of the coefficients are bounded in $\Omega$, the weak maximum principle implies that $\displaystyle \sup_{\Omega} v = \sup_{\partial\Omega} v$.

We may assume that the maximum of $v$ occurs at the boundary point $\xi \in \Gamma$. There are two possibilities that one is $|D_T u|^2(\xi,\tau)< 1$ or $|D_{T}u|^2(\xi,\tau)\geq 1$.

If $|D_T u|^2(\xi,\tau)< 1$, then from $\sin^{2}\alpha(1+|Du|^2)=1+|D_T u|^2$ we see that
$$|Du|^2(\xi,\tau) < \frac{2}{\sin^{2}\alpha_{0}} - 1,$$
and the desired bound is established.

Thus, we may assume that $|D_{T}u|^2(\xi,\tau)\geq 1$. At the point $\xi$ we have
\begin{equation*}
\begin{aligned}
&D_N|Du|^2(\xi,\tau)\leq 0,\\
&D_T|Du|^2(\xi,\tau)=0=D_T v(\xi,\tau),
\end{aligned}
\end{equation*}
Hence, by \eqref{eq2_6}-\eqref{eq2_8} we have at $\xi$,
\begin{equation}\label{ell-eq3_1}
D_T D_N u = \sin\alpha(D_T\alpha)v,
\end{equation}
\begin{equation}\label{ell-eq3_2}
D_N D_T u = \sin\alpha(D_T\alpha)v + k D_T u,
\end{equation}
and
\begin{equation}\label{ell-eq3_3}
D_T D_T u = \frac{\sin\alpha\cos\alpha(D_T\alpha)v^2}{D_T u}.
\end{equation}
We now use an idea from \cite{AW}:
\begin{equation}\label{ell-eq3_4}
(D_N u)(D_N D_N u)(\xi,\tau) + (D_T u)(D_N D_T u)(\xi,\tau)\leq 0
\end{equation}
and derive some relation which contains only first derivatives.

 Using the fact that $\xi$ is a maximum point, \eqref{eq2_2}-\eqref{eq2_4} and \eqref{ell-eq3_1}-\eqref{ell-eq3_4} yield at $\xi$:
\begin{eqnarray*}
v\left(C+H(x,Du)\right) &=&a^{TT}D_T D_T u + a^{TN}D_T D_N u + a^{NT}D_N D_T u + a^{NN}D_N D_N u\\
&-&a^{TT}\langle D_T T,Du \rangle - a^{TN}\langle D_TN,Du \rangle- a^{NT}\langle D_NT,Du \rangle- a^{NN}\langle D_N N,Du \rangle \\
&=&\left(\frac{\cos\alpha\sin\alpha\left(D_T \alpha\right)}{D_T u}\right)\left(v^2 + |D_T u|^2\right) + \frac{2k\cos\alpha |D_T u|^2}{v} \\
\displaystyle &+&\sin^2\alpha D_N D_N u + k\cos\alpha\frac{1+|D_Nu|^2}{v}.
\end{eqnarray*}
That is,
\begin{eqnarray*}
\sin^2\alpha D_N D_N u&=&-\left(\frac{\cos\alpha\sin\alpha\left(D_T \alpha\right)}{D_T u}\right)\left(v^2 + |D_T u|^2\right)\\
&-&\frac{2k\cos\alpha |D_T u|^2}{v} - \frac{k\cos\alpha\left(1+|D_Nu|^2\right)}{v} + v\left(C+H(x,Du)\right).
\end{eqnarray*}
Using \eqref{ell-eq3_3} and substituting the expression above for $D_N D_N u$ into \eqref{ell-eq3_4} we obtain
\begin{equation}\label{ell-eq3_5}
\begin{aligned}
& \cos^2\alpha \sin\alpha (D_T \alpha)v\left(\frac{v^2+|D_T u|^2}{D_T u}\right)\\
&+ 2k |D_{T}u|^2 \cos^2\alpha + k\cos^2\alpha(1+|D_N u|^2)- \cos\alpha v^2(C+H(x,Du))\\
&+\sin^2\alpha\left(\sin\alpha\left(D_T \alpha\right)v + kD_T u\right)D_T u\leq 0.
\end{aligned}
\end{equation}

We may invoke following algebraic identities from \cite{AW}:
\begin{equation}
\label{a1}
\displaystyle 2k |D_{T}u|^2 \cos^2\alpha + k\cos^2\alpha(1+|D_N u|^2)+k\sin^2\alpha|D_T u|^2 = k(v^2 - 1),
\end{equation}
\begin{equation}
\label{a2}\displaystyle \frac{\cos^2\alpha \sin\alpha (D_T \alpha)}{D_T u}v\left(v^2+|D_T u|^2\right)+\sin^3\alpha(D_T \alpha)(D_T u)v = \frac{\sin\alpha(D_T \alpha)}{D_T u}v\left(v^2-1\right),
\end{equation}
and
\begin{equation}
\label{a3}\displaystyle v^2 -1 = \frac{|D_T u|^2}{\sin^2\alpha} + \frac{\cos^2\alpha}{\sin^2\alpha}.
\end{equation}
By these we may simplify \eqref{ell-eq3_5} to give
\begin{equation}\label{ell-eq3_6}
\displaystyle k(v^2 -1) + \frac{D_T u}{\sin\alpha}\left(D_T \alpha\right)v^2 \leq \left(\cos\alpha\right)v^2 (C+H(x,Du)) - \frac{\cos^2\alpha}{\sin\alpha D_T u}\left(D_T \alpha\right)v.
\end{equation}
From \eqref{eq2_3} we see that
$$\displaystyle \left|\frac{D_T u}{\sin\alpha}\right| \leq v.$$
Thus, using $|D_T u| \geq 1$ and $v\geq 1$, we may estimate \eqref{ell-eq3_6} as
\begin{equation}\label{eq3_7}
\left(k - |D_T \alpha|-|C+H(x,Du)|\right)v \leq \frac{|D_T \alpha|}{\sin\alpha} + k.
\end{equation}
By assumption \eqref{ell-eq1_3}, $k - |D_T \alpha|-|C+H(x,Du)| \geq k - |D_T \alpha|-c_0\geq \delta_0$.  Hence,
$$v \leq \delta^{-1}_0\left(k + \frac{k}{\sin\alpha}\right),$$
which is the desired result.
\end{proof}

Using the gradient estimate above and the continuity method as outlined in the introduction we get a smooth solution to BVP \ref{BVP3_1}.

We now make a remark.
One possible method for solving BVP \ref{BVP3_1} is to solve the following perturbation problem for $\varepsilon>0$ small and letting $\varepsilon\to 0$ to find the solution as suggested in \cite{AW}. One may solve the problem
\begin{BVP}\label{BVP3_3}
\begin{equation}\label{eq3_10}
\left\{
\begin{aligned}
& \frac{1}{\sqrt{1+|Dw_\varepsilon|^2}}a^{ij}(Dw_\varepsilon) D_i D_j w_\varepsilon - H(x,Dw_\varepsilon) = \varepsilon w_\varepsilon \ &on&\ \Omega\\
& D_N w_\varepsilon = -\cos\alpha\sqrt{1+|Dw_\varepsilon|^2} \ &on&\ \partial\Omega.
\end{aligned}
\right.
\end{equation}
\end{BVP}

We conjecture that the following result is true.
\begin{conjecture}\label{thm3_5} Assume $\Omega$ convex and $\alpha\in (0,\pi)$. Let $u$ be the solution to
the problem BVP \ref{BVP3_1}. Then, up to a constant, the smooth solution $u$ is the suitable limit from the solution sequence of the problem \ref{BVP3_3} as $\epsilon\to 0$.
\end{conjecture}

The reason is below.
By \cite{Jo,Ge}, we know that solutions to the BVP \ref{BVP3_3} exist for $\varepsilon >0$ and by \cite{K} these solutions are convex ones. If we can show that $\exists c_0$, independent of $\varepsilon$, such that $|\varepsilon w_\varepsilon|^2 \leq c_0$, then we may replace $u_t$ with $\varepsilon w_\varepsilon$ in the gradient estimate of Theorem \ref{ell-thm3-2} and conclude that a limit solution exists for $\varepsilon \rightarrow 0$. In \cite{LTU}, since $\varepsilon\rightarrow0$ implies $|D(\varepsilon w_\varepsilon)|^2 \rightarrow 0$, we know $\varepsilon w_\varepsilon \rightarrow C$, which gives us a solution to \eqref{eq3_8}.

 We borrow the idea from \cite{AW} and we let $\psi$ be a smooth function on $\Omega$ satisfying $D_N \psi < -\cos\alpha \sqrt{1+|D\psi|^2}$ on $\partial\Omega$ and let $A$ be a constant such that $A<-\cos\alpha\sqrt{1+A^2}$ on $\partial\Omega$. Then a function $\psi$ defined to be $\psi = Ad$ near the boundary and extended to be a smooth function on all of $\Omega$ would satisfy our requirements. In any case, for some such choice of $\varphi$, let  $\xi \in \overline{\Omega}$ be a point where $\psi - w_\varepsilon$ has its minimum.\\
\indent If $\xi \in \partial\Omega$, then $D_T \psi(\xi) = D_T w_\varepsilon(\xi)$ and $D_N \psi(\xi) \geq D_N w_\varepsilon(\xi)$. Now, for a fixed constant $a$, the monotonicity in $q$ of the function $\displaystyle \frac{q}{\sqrt{1+a^2 +q^2}}$ gives
$$-\cos\alpha(\xi) > \frac{D_N \psi}{\sqrt{1+|D\psi|^2}}\left(\xi\right) \geq \frac{D_N w_\varepsilon}{\sqrt{1+|Dw_\varepsilon|^2}}\left(\xi\right)= -\cos\alpha(\xi),$$
which is a contradiction.\\
\indent Thus $\xi \in \Omega$ and $D\psi(\xi) = Dw_\varepsilon(\xi)$ and $D^2 \psi(\xi) \geq D^2 w_\varepsilon(\xi)$. Therefore, there exists a constant $c=c(\psi)$ such that
$$c \geq \frac{1}{\sqrt{1+|D\psi|^2}}a^{ij}(D\psi)D_i D_j \psi(\xi)- H(x,D\psi)  \geq \frac{1}{\sqrt{1+|Dw_\varepsilon|^2}}a^{ij}(D w_\varepsilon)D_i D_j w_\varepsilon(\xi)- H(x,Dw_\varepsilon)  = \varepsilon w_\varepsilon(\xi).$$
Hence, $\varepsilon\psi(x) - \varepsilon w_\varepsilon(x) \geq \varepsilon\psi(\xi) - \varepsilon w_\varepsilon(\xi) \geq\varepsilon\psi(\xi) - c$, i.e.,
\begin{equation}\label{eq3_11}
\varepsilon w_\varepsilon(x) \leq \varepsilon\psi(x) - \varepsilon \psi(\xi) + c.
\end{equation}
Thus, we have proved the upper bound of $\varepsilon w_\varepsilon$. Similarly, we can prove the lower bound of $\varepsilon w_\varepsilon$. Assume that $\exists B$ such that $B>-\cos\alpha(x)\sqrt{1+B^2}$ on $\partial\Omega$. Then, we can find a $\psi$ of the form $\psi = Bd$ with $\psi$ satisfying $D_N \psi>-\cos\alpha\sqrt{1+|D\psi|^2}$. For some such choice of $\psi$, let $\xi\in\overline{\Omega}$ be a point where $\psi - w_\varepsilon$ has its maximum.\\
\indent If $\xi \in \partial\Omega$, then $D_T \psi(\xi) = D_T w_\varepsilon(\xi)$ and $D_N \psi(\xi) \leq D_N w_\varepsilon(\xi)$. We then have
$$-\cos\alpha(\xi) < \frac{D_N \psi}{\sqrt{1+|D\psi|^2}}\left(\xi\right) \leq \frac{D_N w_\varepsilon}{\sqrt{1+|Dw_\varepsilon|^2}}\left(\xi\right) = -\cos\alpha(\xi),$$
which is a contradiction.

Thus $\xi \in \Omega$ and $D\psi(\xi) = Dw_\varepsilon(\xi)$ and $D^2 \psi(\xi) \leq D^2 w_\varepsilon(\xi)$. Therefore, there exists a constant $c'=c'(\psi)$ such that
$$c'\leq \frac{1}{\sqrt{1+|D\psi|^2}}a^{ij}(D\psi)D_i D_j \psi(\xi)- H(x,D\psi) \leq \frac{1}{\sqrt{1+|Dw_\varepsilon|^2}}a^{ij}(D w_\varepsilon)D_i D_j w_\varepsilon(\xi)- H(x,Dw_\varepsilon) = \varepsilon w_\varepsilon(\xi).$$
Hence, $\varepsilon\psi(x) - \varepsilon w_\varepsilon(x) \leq \varepsilon\psi(\xi) - \varepsilon w_\varepsilon(\xi)$ implies
\begin{equation}\label{eq3_12}
\varepsilon w_\varepsilon(x) \geq \varepsilon\psi(x) - \varepsilon \psi(\xi) + c'.
\end{equation}
By \eqref{eq3_11} and \eqref{eq3_12}, we can prove $|\varepsilon w_\varepsilon| \leq c_2$ where $c_2 = max\{|c|,|c'|\}$.

We now prove the uniqueness of the solution of BVP \ref{BVP3_3}. Assume that there exist two solutions $w_1$ and $w_2$ solving \eqref{eq3_8} with $C_1$ and $C_2$ on the right hand side and $C_1 < C_2$. Without loss of generality, via a translation by a constant, we may assume that $w_1 \geq w_2$. For $w = w_1 - w_2$, we have
$$L(w)= div \left( \frac{Dw_1}{\sqrt{1+|Dw_1|^2}}\right)- H(x,Dw_1) - div \left( \frac{Dw_2}{\sqrt{1+|Dw_2|^2}} \right)+H(x,Dw_2) = C_1 - C_2 < 0.$$
Note that $H(x,Dw_1)-H(x,Dw_2)=\int_0^1H_p(x, \nabla w_2+(1-t)\nabla w_1))dt \nabla w$,
 and for $F(X)=\frac{X}{\sqrt{1+|X|^2}}$, we have
$$
F(\nabla w_1)-F(\nabla u_2)=\int_0^1(dF(t(\nabla w_2+(1-t)\nabla w_1))dt \nabla w.
$$
Thus, $w$ satisfies a linear elliptic differential inequality $L(w) < 0$ and it follows from the maximum principle that the minimum of $w$ must occur at $\xi \in \partial\Omega$. Then $|D_T w_1|^2 (\xi) = |D_T w_2|^2 (\xi) = a^2$ for some $a\in \mathbb{R}^{+}$. Since both solutions satisfy the same boundary conditions,
$$\displaystyle \frac{D_N w_1}{\sqrt{1+a^2 + |D_N w_1|^2}} (\xi) = \frac{D_N w_2}{\sqrt{1+a^2 + |D_N w_2|^2}} (\xi).$$
Again, we may use the strict monotonicity in $q$ of the function $\displaystyle \frac{q}{\sqrt{1+a^2+q^2}}$ to conclude $D_N w_1 = D_N w_2$. But $D_N w = 0$ yields a contradiction to the Hopf boundary point lemma. Thus, $C_1 \geq C_2$.\\
\indent By reversing the roles of $w_1$ and $w_2$ we may obtain the opposite inequality. Thus, $C_1 = C_2$.
Then we may use the strong maximum principle to know that $w_1 = w_2$ up to a constant.

For the contact angle $\alpha\in (0,\pi)$, letting $w=w_\varepsilon$ and $v=\sqrt{1+|Dw|^2}$, we may compute that
$$
\varepsilon w_kw_{k}+w_kH_{x_k}(x,Dw_\varepsilon)+w_kH_p(x,Dw_\varepsilon)Dw_k = w_k\left[div\left(\frac{Dw}{\sqrt{1+|Dw|^2}}\right)\right]_k,
$$
and
\begin{eqnarray*}
 \varepsilon |Dw|^2+w_kH_{x_k}(x,Dw_\varepsilon)+\frac12H_p(x,Dw_\varepsilon)D|Dw|^2  &=& {w_k}\left[\left( \frac{w_i}{\sqrt{1+|Dw|^2}}\right)_i\right]_k\\
&=&  \left( \frac{1}{v}a^{ij} w_k w_{kj}\right)_i - \frac{1}{v}a^{ij}w_{ki}w_{kj}\\
&=&  \left( a^{ij} v_j\right)_i - |A|^2v.
\end{eqnarray*}
Hence, we know by the maximum principle that $v$ attains its maximum at the boundary. It seems to us that it is difficult to get that
 as in \cite{LTU},$\exists c > 0$ such that $\varepsilon|Dw_\varepsilon|\leq c$ on $\bar{\Omega}$. If this is true, one may get the limit of the solution  sequence and the existence of the solution of BVP \ref{BVP3_1}.

\section{Time derivative estimate and gradient estimates}\label{sect4}
In this section we study the uniform estimate for the problem \eqref{eq1_1}
We remark that the local solution to the problem \eqref{eq1_1} can be obtained as in \cite{Ec}.
By using the maximum principle, we can get the following result.
\begin{lemma}\label{lem4_1} For the solution $u$ to the problem \eqref{eq1_1} on $Q_{T}$, we have
$\displaystyle \sup_{Q_{T}}u_{t}=\sup_{\Omega_{0}}u_{t}$, $\displaystyle \inf_{Q_{T}}u_{t}=\inf_{\Omega_{0}}u_{t}$. That is, for $ c_{0}= max_{\bar{\Omega}}|u_t({0})|$ such that $\forall(x,t) \in Q_{T}$,
$\displaystyle |u_{t}|^2(x,t)\leq c_{0}^2$.
\end{lemma}

\begin{proof}
We first show that the maximum must occur on the parabolic boundary $\Gamma_{T} \bigcup \Omega_{0}$. Recall that the equation \eqref{eq1_1} is
$$\displaystyle u_{t} = div\left(\frac{Du}{\sqrt{1+|Du|^2}}\right) -H(x,Du).$$
Then, taking the t-derivative for both sides, we have
\begin{eqnarray*}
\displaystyle \frac{\partial u_{t}}{\partial t}+H_p(x,Du)Du_t&=&\partial_{i}\left[\left(\frac{u_{i}}{\sqrt{1+|Du|^2}}\right)_{t}\right]\\
\displaystyle &=&\partial_{i}\left(\frac{u_{it}\sqrt{1+|Du|^2} - u_{i}\left(\sqrt{1+|Du|^2}\right)_{t}}{1+|Du|^2}\right)\\
\displaystyle &=&\partial_{i}\left(\frac{1}{v}\left(\delta_{ij} - \frac{u_{i}u_{j}}{1+|Du|^2}\right)\left(u_{t}\right)_{j}\right)\\
\displaystyle &=&\partial_{i}\left(\frac{1}{v}a^{ij}\left(u_{t}\right)_{j}\right).
\end{eqnarray*}

Since the coefficients $\displaystyle \frac{1}{v}a^{ij}$ satisfy the uniform ellipticity condition in $Q_T$.
We can apply the weak maximum principle of the second order parabolic equation to obtain $\displaystyle \sup_{Q_T} u_{t} = \sup_{\Gamma_T \cup \Omega_{0}}u_{t}$. Next we explore the possibility that maximum occurs at $(\xi,\tau)\in\Gamma_{T}$. Since $\displaystyle \max|u_{t}|^2 = |u_{t}|^2(\xi,\tau)>0$, $(D_{T}u_{t})(\xi,\tau) = 0$, we have from \eqref{eq2_5} that
\begin{eqnarray*}
(D_{N}u_{t})(\xi,\tau)&=&-\cos\alpha\frac{D_{N} u D_{N} u_{t} + D_{T} u D_{T} u_{t}}{\sqrt{1+|Du|^2}}(\xi,\tau)\\
&=&-\cos\alpha\frac{D_{N}u D_{N}u_{t}}{\sqrt{1+|Du|^2}}(\xi,\tau)\\
&=&\cos^{2}\alpha(D_{N}u_{t})(\xi,\tau).
\end{eqnarray*}
That is, $\displaystyle \sin^{2}\alpha(D_{N}u_{t})(\xi,\tau) = 0$. Since $\displaystyle \alpha \in (0,\pi)$, we have $\sin^{2}\alpha\not=0$ and then $(D_{N}u_{t})(\xi,\tau) = 0$. However, $D_{N}u_{t} = 0$ yields a contradiction to the Hopf boundary point lemma. Thus the maximum occurs at $(\xi,\tau) \in \Omega_{0}$.
Similarly, we can apply the weak minimum principle to obtain that the minimum occurs on $\Omega_{0}$.
Then we have $\forall(x,t) \in Q_{T}$, $\displaystyle |u_{t}|^2(x,t)\leq c_{0}^2$, where $c_0 = max |u_0|$.
\end{proof}

\begin{theorem}\label{thm4-2}
Under the assumptions \eqref{eq1_6}, we have some uniform constant $c_{2}=c_{2}(\alpha_{0},\delta_{0},k_{0},u_{0})>0$ such that
$$ \sup_{Q_T} v \leq c_{2}.
$$
\end{theorem}
\begin{proof}
Again, we first show that the maximum must occur on $\Gamma_{T}\bigcup\Omega_{0}$. By the evolution equation for $u_k$, we have
\begin{eqnarray*}
\frac{u_k}{v}u_{tk} +\frac{u_k}{v}(H_{x_k}+H_p Du_k)&=& \frac{u_k}{v}\left[div\left(\frac{Du}{\sqrt{1+|Du|^2}}\right)\right]_k,\\
v_t+\frac{u_k}{v}(H_{x_k}+H_p Du_k) &=& \frac{u_k}{v}\left[\left( \frac{u_i}{\sqrt{1+|Du|^2}}\right)_i\right]_k\\
&=& \frac{u_k}{v}\left(\left(\frac{u_i}{\sqrt{1 + |Du|^2}}\right)_k \right)_i\\
&=& \frac{1}{v} u_k \left(\frac{1}{v} a^{ij} u_{kj}\right)_i\\
&=& \frac{1}{v} \left( \frac{1}{v}a^{ij} u_k u_{kj}\right)_i - \frac{1}{v^2}a^{ij}u_{ki}u_{kj}\\
&=& \frac{1}{v} \left( a^{ij} v_j\right)_i - |A|^2.
\end{eqnarray*}
Since all of the coefficients are bounded in $Q_T$, the weak maximum principle implies that $\displaystyle \sup_{Q_T} v = \sup_{\Gamma_T \cup \Omega_0} v$.
Next we assume that the maximum of $v$ occurs at $(\xi,\tau) \in \Gamma_{T}$. There are two possibilities. Note that, if $|D_T u|^2(\xi,\tau)< 1$, then from $\sin^{2}\alpha(1+|Du|^2)=1+|D_T u|^2$ we see that
$$|Du|^2(\xi,\tau) < \frac{2}{\sin^{2}\alpha_{0}} - 1,$$
and the bound is established.\\
Thus, we may assume that $|D_{T}u|^2(\xi,\tau)\geq 1$ and we may argue as in the elliptic case. At $(\xi,\tau)$.
\begin{equation*}
\begin{aligned}
&D_N|Du|^2(\xi,\tau)\leq 0,\\
&D_T|Du|^2(\xi,\tau)=0=D_T v(\xi,\tau),
\end{aligned}
\end{equation*}
Hence, by \eqref{eq2_5}-\eqref{eq2_7} we have at $(\xi,\tau)$,
\begin{equation}\label{eq4_1}
D_T D_N u = \sin\alpha(D_T\alpha)v,
\end{equation}
\begin{equation}\label{eq4_2}
D_N D_T u = \sin\alpha(D_T\alpha)v + k D_T u,
\end{equation}
\begin{equation}\label{eq4_3}
D_T D_T u = \frac{\sin\alpha\cos\alpha(D_T\alpha)v^2}{D_T u}.
\end{equation}
Now our goal is to derive from the following expression some inequality which contains only first derivatives:
\begin{equation}\label{eq4_4}
(D_N u)(D_N D_N u)(\xi,\tau) + (D_T u)(D_N D_T u)(\xi,\tau)\leq 0.
\end{equation}
To accomplish this we use the evolution equations for $u_t$. Using the fact that $(\xi,\tau)$ is a maximum point, \eqref{eq2_2}-\eqref{eq2_4} and \eqref{eq4_1}-\eqref{eq4_4} yield at $(\xi,\tau)$:
\begin{eqnarray*}
v u_t&=&a^{TT}D_T D_T u + a^{TN}D_T D_N u + a^{NT}D_N D_T u + a^{NN}D_N D_N u\\
&-&a^{TT}\langle D_T T,Du \rangle - a^{TN}\langle D_TN,Du \rangle- a^{NT}\langle D_NT,Du \rangle- a^{NN}\langle D_N N,Du \rangle - Hv\\
&=&\left(\frac{\cos\alpha\sin\alpha\left(D_T \alpha\right)}{D_T u}\right)\left(v^2 + |D_T u|^2\right) + \frac{2k\cos\alpha |D_T u|^2}{v} \\
\displaystyle &-&\sin^2\alpha D_N D_N u + k\cos\alpha\frac{1+|D_Nu|^2}{v} - Hv.
\end{eqnarray*}
That is,
\begin{eqnarray*}
\sin^2\alpha D_N D_N u&=&v u_t -\left(\frac{\cos\alpha\sin\alpha\left(D_T \alpha\right)}{D_T u}\right)\left(v^2 + |D_T u|^2\right)\\
&-&\frac{2k\cos\alpha |D_T u|^2}{v} - \frac{k\cos\alpha\left(1+|D_Nu|^2\right)}{v}+Hv.
\end{eqnarray*}
Using \eqref{eq4_3} and substituting the expression above for $D_N D_N u$ into \eqref{eq4_4} we obtain
\begin{equation}\label{eq4_5}
\begin{aligned}
&-\cos\alpha v^2 u_t + \cos^2\alpha \sin\alpha (D_T \alpha)v\left(\frac{v^2+|D_T u|^2}{D_T u}\right)\\
&+ 2k |D_{T}u|^2 \cos^2\alpha + k\cos^2\alpha(1+|D_N u|^2)\\
&+\sin^2\alpha\left(\sin\alpha\left(D_T \alpha\right)v + kD_T u\right)D_T u - Hv^2 \cos\alpha \leq 0.
\end{aligned}
\end{equation}

We now use \eqref{a1}-\eqref{a3} to simplify \eqref{eq4_5} and we have
\begin{equation}\label{eq4_6}
\displaystyle k(v^2 -1) + \frac{D_T u}{\sin\alpha}\left(D_T \alpha\right)v  -H v^2 \cos \alpha \leq \left(\cos\alpha\right)v^2 u_t - \frac{\cos^2\alpha}{\sin\alpha D_T u}\left(D_T \alpha\right)v.
\end{equation}
From \eqref{eq2_3} we see that
$$\displaystyle \left|\frac{D_T u}{\sin\alpha}\right| \leq v.$$
Thus, using $|D_T u| \geq 1$ and $v\geq 1$, we may estimate \eqref{eq4_6} as
\begin{equation}\label{eq4_7}
\left(k - |D_T \alpha|-|u_t|-|H|\right)v \leq \frac{|D_T \alpha|}{\sin\alpha} + k.
\end{equation}
By assumption \eqref{eq1_6}, $k - |D_T \alpha|-|u_t| -|H| \geq k - |D_T \alpha|-c_1\geq \delta_0$.  Hence,
$$v \leq \delta^{-1}_0\left(k + \frac{k}{\sin\alpha}\right),$$
which is the desired result.
\end{proof}

\section{Asymptotic behavior of the flow \eqref{eq1_1}}\label{sect5}

In this section we assume that $H=H(x,p)=H(x)$ on $\overline{\Omega}\times \mathbb R^2$ is a smooth function such that the elliptic problem \eqref{ell-eq1_2} admits a solution. From now on, we assume that $\alpha$ is the angle function such that elliptic version of \eqref{eq1_1} has a solution
with  $\Omega$ being a bounded convex domain in the plane. This assumption is not strong since we may refer to \cite{Li2} for many existence results.
In fact, in many cases we know that the translating soliton exists \cite{Li2} and we have the result below.
\begin{Corollary}\label{Cor5_1} Let $\alpha\in (0,\pi)$ and let $\Omega$ be a bounded convex domain in the plane.
For a solution $u = u(x,t)$ of \eqref{eq1_1}, $\exists c_3>0$ a uniform constant such that
$$|u(x,t) - Ct| \leq c_3$$
\end{Corollary}
\begin{proof}
\indent For a solution $\tilde{u} = \tilde{u}(x)$ of BVP \ref{BVP3_1}, it is obvious that $u(x,t) = \tilde{u} + Ct$ solves the parabolic problem. That is, $u(x,t)$ is a solution that just translates upwards with speed $C$. This result follows by sandwiching the parabolic solution between two translating elliptic solutions and applying a maximum principle similar to the one used in Theorem \ref{thm3_5}.
\end{proof}
Up to now, we have proved the existence of global solutions to \eqref{eq1_1} with $\alpha=0$. The goal of this section is study the asymptotic behavior of the global solution. Our result is below.
\begin{theorem}\label{thm5_2}
Let $u_1$ and $u_2$ be any two solutions of \eqref{eq1_1} and let $u = u_1 - u_2$. Then $u$ becomes a constant function as $t \rightarrow \infty$. In particular, since $\tilde{w} = Ct + w$ solves \eqref{eq1_1} when $w$ solves BVP \ref{BVP3_3}, all limit solutions of \eqref{eq1_1} are $\tilde{w}$ up to translation.\end{theorem}
\begin{proof} Recall for $F(X)=\frac{X}{\sqrt{1+|X|^2}}$, we have
$$
F(\nabla u_1)-F(\nabla u_2)=\int_0^1(dF(t(\nabla u_2+(1-t)\nabla u_1))\nabla udt.
$$
Then $u$ satisfies a linear parabolic equation
\begin{equation}\label{eq4_1}
\left\{
\begin{aligned}
& \frac{\partial}{\partial t} u = div\left(A(x,t) \nabla u\right) \ &on \ Q_T\\
& 0 = \tilde{c}^{ij} D_{i} u N_{j} \ &on \ \Gamma_T
\end{aligned}
\right.
\end{equation}
where
$$A(x,t) =\int_0^1(dF(t(\nabla u_2+(1-t)\nabla u_1))dt$$
and  $ \tilde{c}^{ij}$ are determined similarly. Note that $A(x,t)$ and $\tilde{c}^{ij}$ are positive definite matrices. The strong maximum principle implies that $osc(t) = max\ u(x,t)-min\ u(x,t) \geq 0$ is a strictly decreasing function in time unless $u$ is constant. Now, if $\displaystyle \lim_{t\rightarrow\infty} u(\cdot,t)$ were not a constant function, then a limit of $u_n(x,t) = u(x,t-t_n)$ as $t_n \rightarrow \infty$ would yield a solution on $\Omega \times (-\infty,+\infty)$ which would not be constant but on which $osc(t)$ would be constant. This, however, would contradict the strong maximum principle.
\end{proof}
It follows that, in the case where the average cosine of the contact angle is zero, solutions converge to minimal surfaces. We now give an explicit derivation of this fact.
\begin{Remark}\label{Rem5_3}
If $C = 0$ for the elliptic problem BVP \ref{BVP3_1}, $\displaystyle \lim_{t \rightarrow\infty}u_t \rightarrow 0$. That is, the solutions converge to the corresponding  surface of mean curvature $H(x)$.
\end{Remark}
\begin{proof}
We prove this remark below. We have
$$\displaystyle\frac{d}{dt}\int_{\Omega} v dx = \int_{\Omega} \frac{D_i u_t D_i u}{v}dx = - \int_{\Omega} u_t ^2dx + \int_{\partial\Omega} u_t \cos\alpha ds-\int_{\Omega} u_tH(x),$$
which implies
$$\displaystyle \frac{d}{dt}\left(\int_{\Omega} v dx - \int_{\partial\Omega} u\cos\alpha ds+\int_{\Omega} uH(x)\right) = - \int_{\Omega} u_t^2 dx.$$
Since $C=\displaystyle \int_{\Omega} \cos\alpha ds = 0$ and $u$ is uniformly bounded, Theorem \ref{thm5_2} and Corollary \ref{Cor5_1} imply $\exists c_3 \in \mathbb{R}^{+}$ such that
$$\displaystyle \int^{\infty}_{0}\int_{\Omega} u_t^2 dx dt \leq c_3.$$
One may then apply standard estimates to conclude that $\displaystyle \lim_{t \rightarrow \infty} u_t \rightarrow 0$. That is to say, the limit is the surface of mean curvature $H(x)$.
\end{proof}


\end{document}